\documentclass[leqno,11pt]{article}%,draft
\usepackage[english]{babel}
\usepackage{amsmath,amsthm,fullpage}
\usepackage{amssymb}
\usepackage{enumerate,bbm}
\usepackage{hyperref}
\newtheorem{theorem}{Theorem}[section]

\newtheorem{lemma}[theorem]{Lemma}

\newtheorem{definition}[theorem]{Definition}
\newtheorem{remark}[theorem]{Remark}
\newcommand{\R}{\mathbbm{R}}		
\newcommand{\N}{\mathbbm{N}}		
\newcommand{\un}{u^{\nu}}		
\newcommand{\vn}{v^{\nu}}		
\newcommand{\vne}{v^{\nu,\epsilon{}}}		

\newcommand{\ue}{u^{E}}
\newcommand{\ve}{v^{E}}\newcommand{\ome}{{\omega}^{E}}

\DeclareMathOperator{\curl}{curl}
\begin{document}
\title{\bf Inviscid limits for the 3D~Navier-Stokes equations with slip
  boundary conditions and applications to the 3D~Boussinesq equations}
\author{Luigi C. Berselli\thanks{Universit\`a degli Studi di Pisa,
    Dipartimento di Matematica, Applicata ``U.~Dini'',
    Via~F.~Buonarroti 1/c, I-56127 Pisa, ITALY, email:
    berselli@dma.unipi.it, URL: http://users.dma.unipi.it/berselli,
    phone:+39 050 2213801, fax:+39 050 2213802}
  \and Stefano Spirito\thanks{Universit\`a dell'Aquila, Dipartimento
    di Matematica, via Vetoio~1, I-67010 Coppito (AQ), ITALY, email:
    stefano.spirito@dm.univaq.it} }

%\date{}
\maketitle
\begin{abstract}
  In this note we consider the inviscid limit for the 3D~Boussinesq
  equations without diffusion, under slip boundary conditions of
  Navier's type. We first study more closely the Navier-Stokes
  equations, to better understand the problem. The role of the initial
  data is also emphasized in connection with the vanishing viscosity
  limit.  

\vspace{.5cm}

\end{abstract}
\section{Introduction}
The aim of this note is to study the $L^\infty(0,T;L^2(\Omega))$
convergence, as $\nu$ vanishes, of (Leray-Hopf) weak solutions to the
3D~Navier-Stokes equations (NSE), towards smooth solutions of the
3D~Euler equations. The dependence of the rate of convergence in terms
of different hypotheses on the initial data is studied and the results
are also applied to handle the problem of convergence of solution
of the 3D Boussinesq equations to those of the Euler-Boussinesq. 

We first study the problem with constant density (set for simplicity
equal to one) and then, in the final section, we treat the Boussinesq
equations. In particular, we start by considering the inviscid limit
for the NSE in a bounded domain $\Omega\subset\R^3$ with smooth
boundary $\Gamma:=\partial\Omega\not=\emptyset$.\par
For the reader's convenience we recall that when $\Gamma$ is non
empty,  for the NSE with Dirichlet boundary conditions (and with $\nu>0$)
\begin{equation*}
%  \label{eq:NS-Dirichlet}
  \begin{aligned}
    \partial_{t}\un-\nu\Delta\un+(\un\cdot\nabla)\,\un+\nabla
    p^\nu&=0&\qquad\textrm{ in }\Omega\times(0,T],
    \\
    \nabla\cdot\un&=0&\qquad\textrm{ in }\Omega\times(0,T],
    \\
    \un&=0&\qquad\textrm{ on }\Gamma\times(0,T],
    \\
    \un(0,x)&=u_{0}^\nu&\qquad \textrm{ in }\Omega,
  \end{aligned}
\end{equation*} 
in general one cannot have convergence (even in weak norms) towards
smooth solutions of the Euler equations, even with the same initial
data $u_0^\nu=u_0^E$
\begin{equation}
  \label{eq:Euler}
  \begin{aligned}
    \partial_{t}\ue+(\ue\cdot\nabla)\, \ue+\nabla
    p^{E}&=0&\qquad\textrm{ in }\Omega\times(0,T],
    \\
    \nabla\cdot \ue&=0&\qquad\textrm{ in }\Omega\times(0,T],
    \\
    \ue\cdot n&=0&\qquad\textrm{ on }\Gamma\times(0,T],
    \\
    \ue(0,x)&=u_{0}^E&\qquad\textrm{ in }\Omega,
  \end{aligned}
\end{equation}
see e.g. the review in Constantin~\cite{Con2007} and
Mazzucato~\cite{Maz2008}. In fact, even if both $\ue$ and $\un$ are
very smooth and both exist in $[0,T]$ (for some positive $T$
independent of the viscosity) certain extra-assumptions are needed in
order to show, at least, that
\begin{equation*}
\text{ as } \nu\to0^+\qquad   \un(t)\to \ue(t)\text{ in }L^2(\Omega),
  \quad\text{uniformly in $t\in[0,T]$}.
\end{equation*}
%where the smooth solution of the Euler equation starting from $u_0$
%exists in $[0,T]$. 
Some necessary and sufficient conditions, related with the dissipation
of energy in a boundary-strip of width depending on $\nu$, have been
detected by Kato~\cite{Kat1984a}. See also recent developments in
Temam and Wang~\cite{TW1997}, Wang~\cite{Wan2001}, and
Kelliher~\cite{Kel2007}. The lack of convergence is due to the
\textit{boundary layer} created from the difference between the
tangential velocity of the Navier-Stokes solution and that of the
Euler solution at the boundary: The first vanishes, while we do not
have control on the tangential velocity of the Euler equations. \par
Better results can be obtained in the case of the NSE with Navier's
boundary conditions.  In Iftimie and Planas~\cite{IP2006} it is
considered the following initial-boundary value problem
\begin{equation}
  \label{eq:NS-Navier-vero}
  \begin{aligned}
    \partial_{t}\un-\nu\Delta\un+(\un\cdot\nabla)\,\un+\nabla
    p^\nu=0&\qquad\textrm{ in }\Omega\times(0,T],
    \\
    \nabla\cdot\un=0&\qquad\textrm{ in }\Omega\times(0,T],
    \\
    \un\cdot n=0&\qquad\textrm{ on }\Gamma\times(0,T],
    \\
  [\mathcal{D}(\un)\,n+\beta\,\un]_{\text{tan}}=0&\qquad\textrm{ on
  }\Gamma\times(0,T], 
    \\
    \un(0,x)=u_{0}^\nu&\qquad \textrm{ in }\Omega,
  \end{aligned}
\end{equation} 
where $\mathcal{D}(\un) = \frac{1}{2} \big[\nabla\un + (\nabla
\un)^T\big] $ is the deformation tensor, $\beta\geq0$ is a constant
(the friction coefficient) and
$[\mathcal{D}(\un)\,n+\beta\,\un]_{\text{tan}}$ is the tangential
component of the vector $\mathcal{D}(\un)\,n+\beta\,\un$. This system
is very close to that originally proposed by Navier~\cite{Nav1823} and
studied analytically (in the stationary case) starting from Solonnikov
and {\v{S}}{\v{c}}adilov~\cite{SS1973}. In particular, the Navier's
slip conditions read as
$[\nu\mathcal{D}(\un)\,n+\beta\,\un]_{\text{tan}}=0$, hence
in~\cite{IP2006} the authors are implicitly assuming the Maxwell
scaling~\cite{Max1879}, with the friction parameter $\beta$ depending
linearly on the viscosity. More details on the role of Navier's
boundary conditions especially for numerical simulations, and some of
the crucial differences between the two dimensional and three
dimensional case, can be found in the review paper~\cite{Ber2010b}.
In the 2D setting the problem is slightly less-hard and classical
results employing slip boundary conditions are those of
Yudovich~\cite{Yud1963}, J.-L.~Lions~\cite{Lio1969}, and
Bardos~\cite{Bar1972}. Interesting results in the 2D case are those
in~\cite{LFMNL2008,LFNLP2005}.\par
A recent vanishing viscosity result in the 3D for
system~\eqref{eq:NS-Navier-vero} is the following one, (see Theorem~1
in~\cite{IP2006}).
\begin{theorem}
  Let $\Omega$ be a bounded smooth open set in $\R^3$ and let
  $u_0\in H^3(\Omega)$ a divergence-free vector field tangent to the
  boundary. For each $\nu>0$ consider $\un_0\in L^2(\Omega) $ a
  divergence free vector field tangent to the boundary such that
  $\un_0\to u_0$ strongly in $L^2(\Omega)$, as $\nu\to 0$. Let $\un$ be
  a weak solution of the Navier-Stokes
  equations~\eqref{eq:NS-Navier-vero} with Navier's boundary conditions
  and with initial datum $\un_0$. Let
  $\ue\in C([0,T];H^3(\Omega))$ be the unique solution of the Euler
  equations~\eqref{eq:Euler}, with initial datum $u_0$,  for some $T\in]0,T_{\max}[$, being
  $T_{\max}$ the maximal time of existence of the smooth solution of
  the Euler equations. Then, $\un$
  converges to $\ue$ strongly in $L^\infty(0,T;L^2(\Omega))$, as
  $\nu\to0$.
\end{theorem}
\begin{remark}
  By inspecting the proof, one can observe that, if the initial data
  converge in $L^2(\Omega)$ fast enough, then the same argument
  implies that
  $\sup_{t\in[0,T]}\|\un(t)-\ue(t)\|^2=\mathcal{O}({\nu})$ and
  $\int_0^T\|\nabla\un(\tau)-\nabla\ue(\tau)\|^2\,d\tau=\mathcal{O}(1)$. The
  result here is independent of the  parameter $\beta\geq0$.
\end{remark}

Our aim is to study the convergence under some different
slip-without-friction boundary conditions, involving the vorticity.
More precisely we will study the following initial-boundary value
problem
\begin{equation}
  \label{eq:NS-Navier}
  \begin{aligned}
    \partial_{t}\un-\nu\Delta\un+(\un\cdot\nabla)\,\un+\nabla
    p^\nu&=0&\qquad\textrm{ in }\Omega\times(0,T],
    \\
    \nabla\cdot\un&=0&\qquad\textrm{ in }\Omega\times(0,T],
    \\
    \un\cdot n&=0&\qquad\textrm{ on }\Gamma\times(0,T],
    \\
    \curl\un\times n&=0&\qquad \textrm{ on }\Gamma\times(0,T],
    \\
    \un(0,x)&=u_{0}^\nu&\qquad \textrm{ in }\Omega.
\end{aligned}
\end{equation} 
and we will show how the convergence-rate can be improved.\par
The interest for these \textit{vorticity based} Navier's boundary
conditions is increasing, especially after the recent results by Xiao
and Xin~\cite{XX2007,XX2011}, Beir\~ao da Veiga and
Crispo~\cite{BeiC2009b,BeiC2010a,BeiC2010b,BeiC2009a}
and~\cite{BS2010b} concerning strong solutions and strong
convergence. See also the related work by
Xin~\textit{et~al.}~\cite{WWX2010,WX2009,XXW2009}.  Next, we point out
that the connection between the two Navier's type conditions,
\eqref{eq:NS-Navier-vero} \textit{versus} \eqref{eq:NS-Navier}, is
expressed by the following identity, valid for all tangential vectors
$\tau$ on the boundary $\Gamma$:
\begin{equation}
  \label{eq:vector_identity}
  {t}\cdot  {\tau}=\frac{\nu}{2}(\curl u\times  n)\cdot  {\tau}-
  \nu\, u\cdot  \frac{\partial n}{\partial {\tau}}\qquad \text{on }\Gamma.
\end{equation}
Here $t$ is the Cauchy stress vector defined by
\begin{equation*}
  {t}(u,p):=n\cdot\mathbbm{T}(u,p)=\sum_{k=1}^n \mathbbm{T}_{i k}(u,p)\,n_k, 
\end{equation*}
with $\mathbbm{T}(u,p):=-\mathrm{I}\,p+\nu\, \mathcal{D}(u)$.  The
vector identity~\eqref{eq:vector_identity} valid on $\Gamma$ shows
that the two Navier's conditions are essentially the same in the case
of a domain with flat boundary. Moreover, in a general domain they
differ by a lower order term.\par
%
%In particular, being the boundary conditions we consider
%in~\eqref{eq:NS-Navier} related with the vorticity, we can have some
%special convergence in the case of ``\textit{well-prepared}'' initial
%data. Similar results have been also recently obtained by Xiao and
%Xin~\cite{XX2011}, while the results presented here are part of the
%phd thesis of the second author, completed at the end of 2011.
%
%We first show in this particular simple situation how the initial
% datum can affect the convergence, by using the hyperbolic nature of
% the limit problem and we study then the Boussinesq equations. We use
% the same observation made in~\cite{BS2010b} to obtain strong
% convergence, by employing in a substantial manner the evolution
% equation for the vorticity, following the evolution along
% pathlines. point out that the approach followed in~\cite{XX2011} is
% slightly different and focusing more on the $H^1(\Omega)$
% convergence and on the case of well-prepared initial data for NSE
% and Euler equations with the same smooth initial datum
% $u_0^\nu=u_0^E\in H^3(\Omega)$. Moreover, in~\cite{XX2011} the
% boundary value problem for the Navier-Stokes equations is a
% non-standard one, slightly different from~\eqref{eq:NS-Navier}. \par
%%
Results similar to the present ones have been also recently obtained
by Xiao and Xin~\cite{XX2011}, while the results presented here are
part of the Ph.D. thesis of the second author, completed at the end of
2011.  We point out that the approach in~\cite{XX2011} is slightly
different, focusing more on $H^1(\Omega)$ convergence in the case of
``well-prepared'' initial datum $u_0^\nu=u_0^E\in H^3(\Omega)$ for
both the NSE and Euler equations. In addition, the boundary value
problem for the Navier-Stokes equations is a non-standard one,
slightly different from~\eqref{eq:NS-Navier}. On the other hand, the
aim of our result is to show how the initial data affect the
convergence of the vanishing viscosity even in the energy
norm. Moreover, we give a simple proof of the the convergence in the
case of well-prepared initial data. In particular, the main result of
this paper is Theorem~\ref{thm:2}, which shows improved convergence in
the $L^2$-norm and convergence also of first derivatives for the
system~\eqref{eq:NS-Navier}, when the initial datum has vorticity
vanishing at the boundary. Observe that in Theorem~\ref{thm:1} the
vanishing vorticity is requested only for the initial datum $u_0^E$ of
the Euler equations. The initial data $u_0^\nu$ of the NSE are just
divergence-free vector fields, tangential to the boundary, and
converging merely in $L^2(\Omega)$ (hence without any control on the
vorticity) to the datum of the Euler equations.

In the final section we use the results obtained for the NSE
equations to tackle the following problem: We study the convergence of
$(\vn,\rho^\nu)$, solution of the viscous Boussinesq equations
\begin{equation}
  \label{eq:NSB-Navier}
  \begin{aligned}
    \partial_{t}\vn-\nu\Delta\vn+(\vn\cdot\nabla)\,\vn+\nabla
    q^\nu&=-\rho^\nu e_3&\qquad\textrm{ in
    }\Omega\times(0,T],
    \\
    \partial_{t}\rho^\nu+(\vn\cdot\nabla)\,\rho^\nu&=0&\qquad\textrm{ in    }\Omega\times(0,T],
    \\
    \nabla\cdot\vn&=0&\qquad\textrm{ in }\Omega\times(0,T],
    \\
    \vn\cdot n&=0&\qquad\textrm{ on }\Gamma\times(0,T],
    \\
    \curl\vn\times n&=0&\qquad \textrm{ on }\Gamma\times(0,T],
    \\
    \vn(0,x)&=v_{0}^\nu&\qquad \textrm{ in }\Omega,
    \\
    \rho^\nu(0,x)&=\rho_{0}^\nu&\qquad \textrm{ in }\Omega,
\end{aligned}
\end{equation} 
toward those of the Euler-Boussinesq equations
\begin{equation}
  \label{eq:EulerB}
  \begin{aligned}
    \partial_{t}\ve+(\ve\cdot\nabla)\,\ve+\nabla q^E&=-\rho^E
    e_3&\qquad\textrm{ in }\Omega\times(0,T],
    \\
    \partial_{t}\rho^E+(\ve\cdot\nabla)\,\rho^E&=0&\qquad\textrm{ in
    }\Omega\times(0,T],
    \\
    \nabla\cdot\ve&=0&\qquad\textrm{ in }\Omega\times(0,T],
    \\
    \ve\cdot n&=0&\qquad\textrm{ on }\Gamma\times(0,T],
    \\
    \ve(0,x)&=v_{0}^E&\qquad \textrm{ in }\Omega,
    \\
    \rho^E(0,x)&=\rho_{0}^E&\qquad \textrm{ in }\Omega,
\end{aligned}
\end{equation} 
in the energy space $L^\infty(0,T;L^2(\Omega))$, where $e_3$ is the
third vector of the canonical basis in $\R^3$.

\medskip

\noindent\textbf{Plan of the paper} In Section~\ref{sec:preliminaries}
we briefly recall the notation, some vector identities, and the
existence results for the NSE~\eqref{eq:NS-Navier}. Next, in
Section~\ref{sec:convergence} we prove two different vanishing
viscosity results for the NSE, showing the critical dependence on the
initial datum. Finally, in Section~\ref{sec:Boussinesq} we study the
vanishing viscosity limit for the Boussinesq system.
\section{Preliminaries}
\label{sec:preliminaries}
We consider a bounded domain $\Omega\subset \R^3$ with smooth boundary
$\Gamma$, say of class $C^4$, and $n$ denotes the exterior normal unit
vector on $\Gamma$.  We will use the classical Lebesgue spaces
$(L^2(\Omega),\|\,.\,\|)$ and $(L^2(\Gamma),\|\,.\,\|_{\Gamma})$ and
the Sobolev spaces $(H^{k}(\Omega),\|\,.\,\|_{k})$ for $k\in\N$ (we do
not distinguish between scalar and vector valued functions). We will
denote by $(H^{s}(\Gamma),\|\,.\,\|_{s,\Gamma})$ the standard trace
spaces on the boundary $\Gamma$. We will also denote by $C$ generic
constants, which may change from line to line, but which are
independent of the viscosity and of the solution of the equations we
are considering.\\
\\
We first start by recalling  the precise notion of weak solution for the
NSE we will use.
\begin{definition}%[Definition of weak solution of~\eqref{eq:NS-Navier}]
\label{def:weak-solution}
We say that $\un\in L^{\infty}(0,T;L^{2}(\Omega))\cap
L^{2}(0,T;H^{1}(\Omega))$, weakly divergence-free and tangential to
the boundary, is a (Leray-Hopf) weak solution of the Navier-Stokes
equations~\eqref{eq:NS-Navier} if the two following conditions hold:
\begin{equation}
%  \label{eq:ws1}
  \begin{aligned}
    %\int_{\Omega}\un(T)\phi(T)\,dx d\tau+
    \int_{0}^{T}\int_{\Omega}\big( -\un\phi_{t}+\nu\nabla \un\nabla
    \phi-(\un\cdot\nabla)\, \phi\,\un\big)\,dx d\tau 
%    \\
    +\nu\int_{0}^{T}\int_{\Gamma}
    \un\cdot (\nabla n)^T\cdot \phi\,dS d\tau
    \\
=\int_{\Omega}u_{0}^\nu\phi(0)\,dx,
  \end{aligned}
\end{equation}
for all vector-fields $\phi\in
C^{\infty}_0([0,T[\times\overline{\Omega})$ such that
$\nabla\cdot\phi=0$ in $\Omega\times[0,T[$, and $\phi\cdot n=0$ on
$\Gamma\times[0,T[$; the following energy estimate
\begin{equation}
  \label{eq:Energy-NS-NS} 
%  \frac{1}{2}  \int_{\Omega}|\un(t)|^{2}\,dx
%  +\nu\int_{0}^{t}\int_{\Omega}|\nabla\un|^{2}\,dx d\tau+
%  \nu\int_{0}^{t}\int_{\Gamma}\un\cdot(\nabla 
%  n)^T\cdot\un\,dS
%  d\tau\leq\frac{1}{2}\int_{\Omega}|u_{0}^\nu|^{2}\,dx,  
\frac{1}{2}  \|\un(t)\|^{2}
  +\nu\int_{0}^{t}\|\nabla\un\|^{2}\,d\tau+
  \nu\int_{0}^{t}\int_{\Gamma}\un\cdot(\nabla 
  n)^T\cdot\un\,dS d\tau\leq\frac{1}{2}\|u_{0}^\nu\|^{2,}
\end{equation}
is satisfied for all $t\in[0,T]$.
\end{definition}

With this definition we have the following result.
\begin{theorem}
  \label{thm:existence_weak_solutions}
  Let be given any positive $T>0$ and $u_0^\nu\in L^2(\Omega)$ which
  is weakly divergence-free and such that $u_0^\nu\cdot n=0$ on $\Gamma$. Then,
  there exists at least a weak solution $\un$ of the Navier-Stokes
  equations~\eqref{eq:NS-Navier} on $[0,T]$.
\end{theorem}

The proof of global existence of weak solutions in the sense of the
Definition~\ref{def:weak-solution} can be found
for instance in~\cite[\S~6]{XX2007}. We observe now that our definition of energy
inequality is slightly different from that in the above reference and
we explain now the equivalence.  To this end we recall the
following formulas for integration by parts (see Ref.~\cite{BB2009} for
the proof).
\begin{lemma}
  \label{lem:gammalap}
  Let $u$ and $\phi$ be two smooth enough vector fields, tangential to
  the boundary $\Gamma$.  Then it follows
  \begin{equation*}%\label{eq:by_parts00}
      -\int_\Omega \Delta u\,\phi\,dx=\int_\Omega \nabla u\nabla
      \phi\,dx-\int_{\Gamma} (\omega\times
      n)\,\phi\,dS+\int_{\Gamma}u\cdot(\nabla n)^T\cdot\phi\,dS,
  \end{equation*}
 where $\omega=\curl u$. Moreover, if $\nabla\cdot u=0$, then $-\Delta
 u=\curl\curl u$, and   
      \begin{equation*}
\int_\Omega \curl \omega\,\phi\,dx=        -\int_\Omega \Delta u\,\phi\,dx=\int_\Omega \omega(\curl
        \phi)\,dx+\int_{\Gamma} (\omega\times  n)\,\phi\,dS.
    \end{equation*}
\end{lemma}
Constructing weak solutions by the usual Galerkin method as
in~\cite{XX2007}, we get for the approximate solutions
$\{\un_m\}_{m\in\N}$ the following identity:
\begin{equation*}
  \frac{1}{2}  \int_{\Omega}|\un_m(t)|^{2}\,dx+\nu\int_{0}^{t}\int_{\Omega}|\curl\un_m|^{2}\,dx d\tau=
  \frac{1}{2}\int_{\Omega}|u_{0m}^\nu|^{2}\,dx\qquad t\in[0,T],
\end{equation*}
where we used as test function the function $\un_m$ itself and the
second integration by parts formula from Lemma~\ref{lem:gammalap} with
$u=\phi=\un_m$ (here the boundary integral vanishes due to the
vorticity-based Navier's boundary conditions). Then, the usual
compactness tools imply that $\un_m$ converge as $m\to+\infty$ to a
weak solution $\un$ and the lower-semi-continuity of the norm implies
that
\begin{equation*}
%  \label{eq:Energy-NS-NS-Xin}
  \frac{1}{2}  \int_{\Omega}|\un(t)|^{2}\,dx+
  \nu\int_{0}^{t}\int_{\Omega}|\curl\un|^{2}\,dx d\tau
  \leq\frac{1}{2}\int_{\Omega}|u_{0}^\nu|^{2}\,dx,
\end{equation*}
which is the energy inequality~(6.12) in~\cite{XX2007}. Finally, by
using the first integration by parts formula from
Lemma~\ref{lem:gammalap} we get~\eqref{eq:Energy-NS-NS}.  This
inequality will be used later on to let some of the calculations
(which will be otherwise only formal) completely justified.\par
To conclude, we recall a well-known existence theorem for smooth
solutions of the Euler equations~\eqref{eq:Euler} in Sobolev spaces.
\begin{theorem}
  Let be given $u_0^E\in H^3(\Omega)$ such that $\nabla\cdot u_0^E=0$
  and $u_0^E\cdot n=0$ on $\Gamma$. Then, there exists a positive time
  $T=T(\|u_0^E\|_{3})>0$ such that a unique solution
  of~\eqref{eq:Euler} exists in
 \begin{equation*}
%   \label{e1}
   u^E\in C([0,T];H^{3}(\Omega)).
 \end{equation*}
\end{theorem}
The proof in the case of a bounded domain can be found in Ebin and
Marsden~\cite{EM1970} and Temam~\cite{Tem1975}. In particular
$T\geq\frac{C}{\|u_0^E\|_3}$, for some $C>0$ independent of the
solution, and in the sequel $T$ will be any positive time strictly
smaller than the maximal time of existence $T_{\max}$.
\section{Proof of the convergence results}
\label{sec:convergence}
We start by showing the basic convergence result, which is the
counterpart of~\cite[Thm.~1]{IP2006} in our setting. \
\begin{theorem}
  \label{thm:1}
  Let $\Omega$ be a bounded smooth open set in $\R^3$, and let
  $u_0^E\in H^3(\Omega)$, be a divergence-free vector-field tangential
  to the boundary. Let $\ue\in C([0,T];H^3(\Omega))$ be the unique
  solution of the Euler equations~\eqref{eq:Euler}, with initial datum
  $u_0$ and defined in some interval $[0,T]$.  Let $\un$ be a weak
  solution of the NSE~\eqref{eq:NS-Navier} with
  divergence-free and tangential to the boundary initial datum $u_0^\nu\in
  L^2(\Omega)$, and with the vorticity-based Navier's
  conditions. Suppose also that
  \begin{equation*}
    \|u^\nu_0-u_0^E\|=\mathcal{O}(\nu^{\frac{3}{2}}).
  \end{equation*}
Then,
  $\sup_{t\in[0,T]}\|\un(t)-\ue(t)\|^2=\mathcal{O}({\nu}^{\frac{3}{2}})$
  and
  $\int_0^T\|\nabla\un(\tau)-\nabla\ue(\tau)\|^2\,d\tau=\mathcal{O}({\nu}^{\frac{1}{2}})$.
\end{theorem}

\begin{proof}%[Proof of Thm.~\ref{thm:1}]
  The proof is simply obtained by taking the difference of the
  equation satisfied by $\un$ with that for $\ue$, multiplying by
  $u:=\un-\ue$, and integrating by parts over
  $\Omega\times(0,T)$. Unfortunately, this cannot be done in a so
  straightforward manner since $\un$ is a weak solution, hence using
  directly $u$ (having the regularity of $\un$) as test function is
  not allowed. We need to pass to an integral formulation and to use
  the energy inequality~\eqref{eq:Energy-NS-NS} to make the argument
  rigorous. The reader well-acquainted with the argument can go
  directly to the formula~\eqref{eq:p5bis}.\par
  We first observe that since $\ue$ is a smooth solutions of the Euler
  equations~\eqref{eq:Euler} in $[0,T]\times\Omega$, $u^E$ is allowed
  as test function for the NSE. Then, after certain integration by
  parts, we get for all $t\in[0,T]$
  \begin{equation}
    \label{eq:energy-NS-E}
    \begin{aligned}
      \int_{\Omega}\un(t)\ue(t)\,dx+ \int_{0}^{t}\int_{\Omega}\Big(\nu\nabla
      \un\nabla \ue - \un\ue_{t}+(\un\cdot\nabla)\, \un
      \ue\Big)\,dx d\tau
      \\
      +\nu\int_{0}^{t}\int_{\Gamma}\un(\nabla n)^T
      \ue\,dS d\tau=\int_{\Omega}u_{0}^E u_0^\nu\,dx.
    \end{aligned}
  \end{equation}
  A further identity is obtained by multiplying the Euler equations by
  $\un$. Since $\ue$ is a local smooth solution everything is
  well-defined and we get
  \begin{equation}
    \label{eq:energy-E-NS}
    \int_{0}^{t}\int_{\Omega}\big(\ue_t
    \un+(\ue\cdot\nabla)\, \ue \un\big)\,dx d\tau
    =\int_{\Omega}u_{0}^E u_0^\nu\,dx.
  \end{equation}
  Next, by multiplying the Euler equations by $\ue$ and by the usual integrations
  by parts we get the energy conservation
  \begin{equation}
    \label{eq:energy-E-E}
    \frac{1}{2}\int_{\Omega}|\ue(t)|^2\,dx=\frac{1}{2}\int_{\Omega}|u_{0}^E|^2\,dx,
  \end{equation}
  where we used the fact that $\ue$ is smooth, tangential to the
  boundary, and divergence-free.\par
Then, by adding together~\eqref{eq:Energy-NS-NS}-\eqref{eq:energy-E-E}
and subtracting~\eqref{eq:energy-NS-E}-\eqref{eq:energy-E-NS}, we
get
  \begin{equation} 
    \label{eq:p5} 
    \begin{aligned}
      \frac{\|u(t)\|^{2}}{2}&
      +\nu\int_{0}^{t}\int_{\Omega}\nabla\un\nabla u\,dx
      d\tau+\int_{0}^{t}\int_{\Omega}(u\cdot\nabla)\,\ue u\,dx d\tau
      \\
      &\qquad +\nu \int_{0}^{t}\int_{\Gamma}\un(\nabla n)^T u\,dS
      d\tau\leq \frac{\|u(0)\|^{2}}{2}.
     \end{aligned}
     \end{equation}
  Let us consider the second term from the left-hand side
  of~\eqref{eq:p5}. By using the parallelogram equality we get 
  \begin{equation*}
   \begin{aligned}
   \int_{0}^{t}\int_{\Omega}\nabla\un\nabla    u\,dx d\tau=
    \frac{1}{2}\int_{0}^{t}\int_{\Omega}|\nabla\un|^2\,dx d\tau
    &+\frac{1}{2}\int_{0}^{t}\int_{\Omega}|\nabla u|^{2}\,dx d\tau
    \\
    &-\frac{1}{2}\int_{0}^{t}\int_{\Omega}|\nabla \ue|^{2}\,dx d\tau. 
  \end{aligned}
  \end{equation*}
  Then, we estimate the other two terms from the left-hand side
  of~\eqref{eq:p5}. We handle the third term by using the higher-order
  regularity of $\ue$, namely we use $\ue\in
  C([0,T];H^3(\Omega))\hookrightarrow C([0,T];W^{1,\infty}(\Omega))$, to get
  \begin{equation*}
   \left|\int_{0}^{t}\int_{\Omega}(u\cdot\nabla)\,\ue\, u\,dx
      d\tau\right|\leq C\int_{0}^{t}\|u(t)\|^2\,d\tau,
  \end{equation*}
for some $C$ independent of $\nu$.
  Concerning the boundary term in~\eqref{eq:p5}, by using trace theorems
  and Young inequality we get, due to the smoothness of $\Gamma$,
  \begin{equation*}
 \begin{aligned}
   \nu   \left| \int_{0}^{t}\int_{\Gamma}\un(\nabla n)^T u\,dS
     d\tau\right|&\leq \nu\int_{\Gamma}|\un|\,|\nabla n|\,| u|\,dS
   \\
   &\leq C\,\nu \|\un\|_{\Gamma}\|u\|_{\Gamma}
   \\
   &\leq C\,\nu \|\un\|_{1/2}\|u\|_{1/2}
   \\
   &\leq
   C\,\nu\|\un\|^{\frac{1}{2}}\|\nabla\un\|^{\frac{1}{2}}\|u\|^{\frac{1}{2}}\|\nabla
   u\|^{\frac{1}{2}}
   \\
   &\leq \frac{\nu}{4}\|\nabla \un\|^{2}+\frac{\nu}{4}\|\nabla
   u\|^{2}+C\nu\|\un\|\|u\|.
  \end{aligned}
\end{equation*}
In particular, to handle the $H^{1/2}(\Gamma)$-norm and to remove the
zero-order term we have used the fact that for functions tangential to
the boundary the Poincar\'e inequality holds true (see for instance,
Kozono and Yanagisawa~\cite{KY2009}) and also that from the energy
inequality both $\|\un\|$ and $\|u^E\|$ are bounded.

By collecting all the estimates, from~\eqref{eq:p5} we get that
  \begin{equation}
    \label{eq:fundamental1}
    \|u(t)\|^{2}    +{\nu}\int_{0}^{t}\|\nabla
    u(\tau)\|^2\,d\tau\leq     \|u(0)\|^{2} +C\left[\int_{0}^{t}\|u(\tau)\|^2\,d\tau+\nu \right].
  \end{equation}
  We can now use Gronwall lemma, obtaining 
  \begin{equation*}
   \|\un-\ue\|^2_{L^{\infty}(0,T;L^{2}(\Omega))}=\mathcal{O}({\nu})\qquad
    \text{and}\qquad     \|\nabla\un-\nabla\ue\|^2_{L^2(0,T;L^{2}(\Omega))}=\mathcal{O}(1),
  \end{equation*}
  exactly as in~\cite{IP2006}. This is not the result stated in the
  Theorem~\ref{thm:1}, but the calculations are given to better
  understand the differences/improvement. 

\bigskip

  We make now some slightly different manipulations, in order to show
  the better rate-of-convergence stated in Theorem~\ref{thm:1}. To
  this end, in~\eqref{eq:p5} we treat the second term from the
  left-hand-side as follows
  \begin{equation*}
   \int_{0}^{t}\int_{\Omega}\nabla\un\nabla
    u\,dx d\tau=\int_{0}^{t}\int_{\Omega}|\nabla
    u|^{2}\,dx d\tau+\int_{0}^{t}\int_{\Omega}\nabla \ue\nabla
    u\,dx d\tau. 
  \end{equation*}
 We arrive now at (cf.~\eqref{eq:p5})
  \begin{equation} 
    \label{eq:p5bis}
    \begin{aligned}
      \frac{\|u(t)\|^{2}}{2}+\nu\int_{0}^{t}\|\nabla u\|^{2} d\tau
      -\int_{0}^{t}\int_{\Omega}(u\cdot\nabla)\,\ue u\,dx d\tau
      +\nu  \int_{0}^{t}\int_{\Gamma}\un(\nabla n)^T u\,dS d\tau
      \\
      \leq -\nu\int_{0}^{t}\int_{\Omega}\nabla\ue\nabla u\,dx d\tau+\frac{\|u(0)\|^2}{2}
    \end{aligned}
  \end{equation}
  We handle the third term from the left-hand-side exactly as before
  and, by integrating by parts the term from the right-hand-side (by
  using the second identity from Lemma~\ref{lem:gammalap}), we get
  \begin{equation*}
%    \label{p14}
    \begin{aligned}
      -&\nu\int_{0}^{t}\int_{\Omega}\nabla\ue\nabla
      u\,dx d\tau=
      \\
      &\nu\int_{0}^{t}\int_{\Gamma}(\ome\times n)\,
        u\,dS d\tau
      +\nu\int_{0}^{t}\int_{\Gamma}\ue\cdot(\nabla n)^T\cdot
        u\,dS d\tau+\nu\int_{0}^{t}\int_{\Omega}\Delta\ue
        u\,dx d\tau.
    \end{aligned}
  \end{equation*}
 Then, we get the following inequality
\begin{equation*}
   \begin{aligned}
  \frac{\|u(t)\|^{2}}{2}+\nu\int_{0}^{t}\|\nabla u\|^{2}
  d\tau&\leq   \frac{\|u(0)\|^{2}}{2}-\nu\int_{0}^{t}\int_{\Gamma}u\, (\nabla n)^T u\,dS
  d\tau
%-\nu\int_{0}^{t}\int_{\Gamma}\ue\cdot(\nabla n)^T\cdot u\,d Sd\tau
  \\
  &-\int_{0}^{t}\int_{\Omega}(u\cdot\nabla)\,\ue u\,dx
  d\tau+\nu\int_{0}^{t}\int_{\Omega}\Delta\ue u\,dx d\tau
  \\
  &-\nu\int_{0}^{t}\int_{\Gamma}(\ome\times n)\, u\,dS d\tau.
\end{aligned}
\end{equation*}
We estimate the absolute value of the space integral from the
right-hand-side by using Schwartz inequality, trace inequalities, and the
regularity of $\ue$ as follows:
  \begin{align*}
    &   \nu\left|\int_{\Gamma}(\ome\times n) u\,dS\right|\leq
    C\nu\|u\|^{\frac{1}{2}}\|\nabla    u\|^{\frac{1}{2}}
   \leq C\,\nu^{\frac{3}{2}}+C\,\|u\|^{2}+\frac{\nu}{2}\|\nabla u\|^{2},
    \\
    \\
    & \nu\left|\int_{\Gamma} u\cdot( \nabla n)^T\cdot
      u\,dS\right|\leq C \,\nu\|u\|_{\Gamma}^{2}\leq C\,\nu
    \|u\|\|\nabla u\| \leq
    C\,    \nu\|u\|^{2}+\frac{\nu}{2}\|\nabla u\|^{2},
    \\
    \\
    &\nu\left|\int_{\Omega}\Delta \ue u\,dx\right|\leq
    C\big(\|u\|^{2}+\nu^2\big).
\end{align*}  
Then, by using also the energy inequality, we obtain the following
differential inequality (cf. with~\eqref{eq:fundamental1})
  \begin{equation*}
    \|u(t)\|^{2}+\nu\int_{0}^{t}\|\nabla u(\tau)\|^{2}d\tau\leq\|u(0)\|^2+
    C\left[\int_{0}^{t}\|u(\tau)\|^{2}d\tau+\nu^{2}+\nu^{\frac{3}{2}}\right].
  \end{equation*}
  By using Gronwall-Lemma we have that
  \begin{equation*}
%    \label{rate1}
    \|\un-\ue\|^2_{L^{\infty}(0,T;L^{2}(\Omega))}=\mathcal{O}(\nu^{\frac{3}{2}}) \qquad\text{and}\qquad
    \|\nabla\un-\nabla\ue\|^2_{L^{2}(0,T;L^{2}(\Omega))}=\mathcal{O}(\nu^{\frac{1}{2}}),
  \end{equation*}
ending the proof.
\end{proof}

We want now to show better convergence, and this happens if the
initial datum belongs to a particular sub-class. In particular, we use the same
observation made in~\cite{BS2010b,XX2011} to show strong convergence up to
second order derivatives.  We prove now the main result of the paper.
\begin{theorem}
  \label{thm:2}
  Let $\Omega$ be a bounded smooth open set in $\R^3$, and let
  $u_0^E\in H^3(\Omega)$, be a divergence-free vector field tangential to
  the boundary, and such that
  \begin{equation}
    \label{eq:zeroid}
    \omega_{0}^E(x)=0\qquad\qquad \forall\,x\in\Gamma.
  \end{equation}
  Let $\ue\in C([0,T];H^3(\Omega))$ be the unique solution of the Euler
  equations~\eqref{eq:Euler}, with initial datum $u_0^E$ and
  defined in some interval $[0,T]$. 
  Let $\un$ be a weak solution of the Navier-Stokes
  equations~\eqref{eq:NS-Navier-vero} with a divergence free and
  tangential to the boundary initial datum $u_0^n\in L^2(\Omega)$ such that
  \begin{equation*}
    \|u^\nu_0-u_0^E\|=\mathcal{O}(\nu).
  \end{equation*} Then,
  \begin{equation*}
    \sup_{t\in[0,T]}\|\un-\ue\|^2_{L^{\infty}(0,T;L^{2}(\Omega))}=\mathcal{O}(\nu^2),
    \quad\text{and}\quad
    \|\nabla\un-\nabla\ue\|^2_{L^{2}(0,T;L^{2}(\Omega))}=\mathcal{O}(\nu). 
  \end{equation*}
%  $\|\un(t)-\ue(t)\|^2=\mathcal{O}({\nu^2})$ and
%  $\int_0^T\|\un(t)-\ue(t)\|^2=\mathcal{O}({\nu})$. 
\end{theorem}
A critical point in the proof is that of ``\textit{having solution to the Euler
  equations with  vanishing tangential component of the
  vorticity, as the Navier-Stokes equations,}'' that is $\ome\times n=0$ on
$\Gamma\times[0,T]$: In this way one can better estimate the term
$-\int_\Omega\Delta \ue\,u\,dx$ involved in the previous calculations.\par
In general the boundary conditions for the vorticity cannot be enforced
for the Euler equations. In addition, also if the initial datum is
such that $(\ome_0\times n)_{|\Gamma}=0$ this is not enough to have the same
boundary behavior for all positive times. As observed
in~\cite{BeiC2010a}, by using the vorticity equation,
\begin{equation*}
  \ome_t+(\ue\cdot\nabla)\,\ome=(\ome\cdot\nabla)\,\ue,
\end{equation*}
by taking the exterior product with the normal unit vector on
$\Gamma$, and finally by using that $\ome\times n=0$ implies that
$\ome_t\times n=0$ on $\Gamma$, one obtains that an
extra-compatibility condition, \textit{generically false}, should be
satisfied by the initial velocity $\ue_0$.  In particular, this
implies that the Navier's type condition does not \textit{persist} for
positive time and hence excludes the chance of a vanishing-viscosity
limit in topologies such that the vorticity $\ome$ has Sobolev traces
at the boundary. On the other hand, by using the fact that for the
Euler equations the vorticity is transported by the velocity $\ue$ and
stretched by $\nabla \ue$, one can employ the well-known
representation formula for classical solutions
\begin{equation}
  \label{eq:mb}
  \ome(X(\alpha,t),t)=\nabla_\alpha X(\alpha,t)\,\ome(\alpha,0),
\end{equation}
where the path-lines $X:\ \Omega\times [0,T]\to\Omega\subset\R^3$
solve the Cauchy problem
\begin{equation*}
  \begin{cases}
    &\frac{d}{dt}X(\alpha,t)=\ue(X(\alpha,t),t),
    \\
    &X(\alpha,0)=\alpha,
  \end{cases}
\end{equation*}
for $t\in[0,T]$ and for $\alpha\in \Omega$. Since $\ue\cdot n =0$ on
the boundary, the path-lines starting on the boundary remain on the
boundary for all positive times. The fundamental effect for our
studies is the following: Let be given $\overline{\alpha}\in \Gamma$,
then $X(\overline{\alpha},t)\in \Gamma$ for all $t\in [0,T]$ and
consequently
\begin{equation*}
  \ome(X(\overline{\alpha},t),t)\times n =\big[\nabla_\alpha
  X(\overline{\alpha},t)\ome(\overline{\alpha},0)\big]\times
   n ,\qquad \forall\,(\overline{\alpha},t)\in\Gamma\times[0,T]. 
\end{equation*}
Being \textit{generically} the matrix $\nabla_\alpha X(\overline{\alpha},t)$ not a multiple of the
identity, this implies that \textit{generically}
$\ome(X(\overline{\alpha},t),t)\times n \not=0$. In particular
$\ome\times n $ may become non-zero, as soon as $\nabla_a X$ induces a
rotation along any axis not parallel with the normal unit vector
passing through $\overline{\alpha}$.\par
On the other hand, this observation makes clear that it is possible to
have another type of \textit{persistence} for the vorticity, by
suitably restricting the class of initial data. In fact, the same
representation formula with the path-lines shows also the following
result.
\begin{lemma}
  Let $\ue\in C([0,T];H^3(\Omega))$ be the unique solution of the
  Euler equations~\eqref{eq:Euler}. If $\omega_0^E(x)=0$
  \underline{for all} $x\in \Gamma$, then
  \begin{equation*}
    \ome(x,t)=0\qquad\text{\underline{for all} couples }(x,t)\in \Gamma\times[0,T].
  \end{equation*}
\end{lemma}
\begin{remark}
  The class of solutions with vanishing-vorticity at the boundary we
  will employ is in some sense optimal, since in~\cite{BeiC2010b} it
  is shown that if $\omega_0\cdot n\not=0$, then there exists a time
  $\overline{T}>0$ such that the $\ome\times n\not=0$ for $0<t \leq
  \overline{T}$. Moreover this class of initial data in non-empty,
  since smooth and divergence-free functions with compact support
  satisfy the assumptions.
\end{remark}
\begin{remark}
  \label{rem:feequalzero}
  In the case of the Euler equations with a non-zero external force
  $f^E$ the same approach shows that a formula similar
  to~\eqref{eq:mb} holds true:
  \begin{equation*}
    \ome(X(\alpha,t),t)={\nabla_\alpha X(\alpha,t)}\,\ome(\alpha,0)+\int_0^t\curl
    f^E(X(\alpha,\sigma),\sigma)\,d\sigma, 
  \end{equation*}
  hence a sufficient condition to have persistence of the vanishing
  vorticity at the boundary is $\omega_0^E=\curl f^E=0$ on
  $\Gamma$.
\end{remark}

This observation on the persistence of the boundary values for the
vorticity points out that probably the ``$\omega$-based'' boundary
conditions of Navier's-type are much better behaved (in the context of
vanishing-viscosity) than the classical ``$\mathcal{D}(u)$-based''
Navier's ones. The Navier's conditions as in~\eqref{eq:NS-Navier-vero}
involve the symmetric part of the gradient. One can recall that the
symmetric part of the gradient $\mathcal{D}(\ue):=\frac{[\nabla
  \ue+(\nabla \ue)^T]}{2}$ has the following evolution equation
\begin{equation*}
  \frac{D \mathcal{D}(\ue)}{Dt}+\mathcal{D}(\ue)^2+{O}^2(\ue)=-H\pi^E,
\end{equation*}
where as usual $\frac{D }{Dt}$ is the derivative along path-lines,
${O}(\ue):=\frac{[\nabla \ue-(\nabla \ue)^T]}{2}$ is essentially the vorticity
(since ${O}(\ue)h=\frac{1}{2}\omega^E\times h$ for each vector $h$) and
$H\pi^E$ is the Hessian of the pressure.  It seems that (contrary to the
results for $\omega^E$ or equivalently ${O}$) the evolution
of the matrix $\mathcal{D}(\ue)$ cannot be handled, since the pressure does not
disappear and one cannot employ directly a transport/stretching
argument. Hence, the problem related with vanishing-viscosity under
the Navier's boundary conditions seems to require different tools, even
if the friction parameter $\beta$ vanishes. 

\smallskip

\begin{proof}[Proof of Theorem~\ref{thm:2}]
  We can improve a little bit  the rate of convergence of
  Theorem~\ref{thm:1}, by assuming that the initial datum is such that
  the vorticity vanishes at the boundary. In fact, we have seen that
  by assuming~\eqref{eq:zeroid} we have $ \omega^{E}(x,t)=0$ for all
  $(x,t)\in\Gamma\times[0,T]$.  Writing again the same energy
  estimates we employed before to get~\eqref{eq:p5bis}, and by using the identity
  \begin{equation*}
    - \nu\int_{0}^{t}\int_{\Omega}\nabla\ue\nabla    u\,dx d\tau
    =\nu\int_{0}^{t}\int_{\Gamma}\ue\cdot(\nabla n)^T\cdot u\,dS d\tau
    +\nu\int_{0}^{t}\int_{\Omega}\Delta\ue u\,dx d\tau,
 \end{equation*}
 (notice that now the boundary integral $\int_{\Gamma}(\ome\times n)
 u\,dS$ vanishes) we can re-do the same calculations as before
 starting from~\eqref{eq:p5bis} to obtain now
\begin{equation*}
  \|u(t)\|^{2}+\nu\int_{0}^{t}\|\nabla u\|^{2}\,d\tau\leq
  \|u(0)\|^{2}+    C\Big[\int_{0}^{t}\|u(\tau)\|^{2}\,d\tau+\nu^{2}\Big].
  \end{equation*}
  Finally,  by Gronwall's inequality we get 
  \begin{equation*}
    \sup_{t\in[0,T]}\|\un-\ue\|^2_{L^{\infty}(0,T;L^{2}(\Omega))}=\mathcal{O}(\nu^2),\qquad\text{and}\qquad 
    \|\nabla\un-\nabla\ue\|^2_{L^{2}(0,T;L^{2}(\Omega))}=\mathcal{O}(\nu),
  \end{equation*}
concluding the proof.
\end{proof}
\section{Vanishing  viscosity for the 3D~Boussinesq equations}
\label{sec:Boussinesq}
In this section we pass to consider the Boussinesq
system~\eqref{eq:NSB-Navier} and a first step concerns the existence
of weak solutions.  In particular, since the problem is without
diffusion in the equation for the density, a proper notion of weak
solution is needed. Observe that the unknown $\rho^{\nu}$ entering the
equations is not the density, but represents the small variations of
density from the constant state, hence there is no need to assume
$\rho^{\nu}\geq0$.
\begin{definition}
  \label{def:weak-solutionB}
  We say that the pair $\un\in L^{\infty}(0,T;L^{2}(\Omega))\cap
  L^{2}(0,T;H^{1}(\Omega))$ and $\rho^{\nu}\in
  L^{\infty}(0,T;L^{2}(\Omega))$ is a weak solution of the Boussinesq
  equations~\eqref{eq:NSB-Navier} if the two following conditions
  hold:
  \begin{equation*}
%    \label{eq:ws1}
    \begin{aligned}
      \int_{0}^{T}\int_{\Omega}\big(
      -\vn\phi_{t}-\rho^\nu\psi_t+\nu\nabla \vn\nabla \phi
      -(\vn\cdot\nabla)\, \phi\,\vn
      -(\vn\cdot\nabla)\, \psi\,\rho^\nu\big)\,dx d\tau
      \\
      +\nu\int_{0}^{T}\int_{\Gamma}
      \vn\cdot (\nabla n)^T\cdot \phi\,dS d\tau
      =\int_{\Omega}v_{0}^\nu\phi(0)+\rho^\nu_0\psi(0)\,dx,
    \end{aligned}
  \end{equation*}
  for all vector-fields $\phi\in
  C^{\infty}_0([0,T[\times\overline{\Omega})$ such that
  $\nabla\cdot\phi=0$ in $\Omega\times[0,T[$, and $\phi\cdot n=0$ on
  $\Gamma\times[0,T[$ and scalars $\psi\in
  C^{\infty}_0([0,T[\times\overline{\Omega})$
  % 
%Moreover for any $B\in C([0,+\infty[)\cap C^1(]0,\infty[)$ such that
%$B(0)=0$
%\begin{equation*}
%  \partial_tB(\rho)+\nabla\cdot(B(\rho)u)=0.
%\end{equation*}
(In particular the resulting solution satisfies the equations
in the sense of $\mathcal{D}'((0,T)\times \Omega)$); the following
energy estimate
  \begin{equation}
    \label{eq:Energy-NSB-NSB}
    \begin{aligned}
      \frac{\|\vn(t)\|^{2}+\|\rho^{\nu}(t)\|^{2}}{2}+\nu\int_{0}^{t}\|\nabla\vn\|^{2}\,d\tau
      &+\nu\int_{0}^{t}\int_{\Gamma}\vn\cdot(\nabla n)^T\cdot\vn\,dS d\tau\\
      &\leq\frac{\|v_0^\nu\|^{2}+\|\rho_0^{\nu}\|^{2}}{2}.
    \end{aligned}
  \end{equation}
is satisfied for all $t\in[0,T]$.
\end{definition}

With this definition we have the following result.
\begin{theorem}
  \label{thm:existence_weak_solutionsB}
  Let be given any positive $T>0$, $v_0^\nu\in L^2(\Omega)$, and 
  $\rho_0^\nu\in L^2(\Omega)$, such that $v_0^\nu$ is divergence-free
  and such that $v_0^\nu\cdot n=0$ on $\Gamma$. Then, there exists at
  least a weak solution of the Boussinesq
  equations~\eqref{eq:NSB-Navier}  on $[0,T]$.
\end{theorem}
%\begin{remark}[Cf.~Corollary~4.1 in \cite{Fei2004}]
%  Observe that is $\rho\in L^2(0,T:L^2(\Omega))$ solves the continuity
%  equation~$\eqref{eq:NSB-Navier}_2$ for some $u\in
%  L^2(0,T;H^1(\Omega))$, then $\rho$ is a renormalized solution
%\end{remark}
\begin{proof}
  By simplifying a procedure typical of compressible flows, as in
  P.-L.~Lions~\cite{Lio1996} and Feireisl~\cite[\S~4]{Fei2004}, we
  first consider the following approximate system, with $\epsilon>0$
  \begin{equation*}
    \begin{aligned}
      \partial_{t}\vne-\nu\Delta\vne+(\vne\cdot\nabla)\,\vne+\nabla
      p^{\nu,\epsilon}&=-\rho^{\nu,\epsilon} e_3&\qquad\textrm{ in
      }\Omega\times(0,T],
      \\
      \partial_{t}\rho^{\nu,\epsilon}+(\vne\cdot\nabla)\,\rho^{\nu,\epsilon}&
      =\epsilon\Delta\rho^{\nu,\epsilon}& 
      \qquad\textrm{ in }\Omega\times(0,T],
      \\
      \nabla\cdot\vne&=0&\qquad\textrm{ in }\Omega\times(0,T],
      \\
      \vne\cdot n&=0&\qquad\textrm{ on }\Gamma\times(0,T],
      \\
      \curl\vne\times n&=0&\qquad \textrm{ on }\Gamma\times(0,T],
      \\
      n\cdot \nabla\rho^{\nu,\epsilon}&=0&\qquad \textrm{ on
      }\Gamma\times(0,T],
      \\
      \vne(0,x)&=v_{0}^{\nu}&\qquad \textrm{ in }\Omega,
      \\
      \rho^{\nu,\epsilon}(0,x)&=\rho_{0}^{\nu}&\qquad \textrm{ in
      }\Omega.
    \end{aligned}
  \end{equation*} 
  By standard techniques one can show that the above system has at
  least a weak solution
  \begin{equation*}
    \vne,\rho^{\nu,\epsilon}\in L^\infty(0,T;L^2(\Omega))\cap L^2(0,T;H^1(\Omega)),
  \end{equation*}
  satisfying the energy inequality
  \begin{equation*}
    \begin{aligned}
      \frac{\|\vne(t)\|^{2}+\|\rho^{\nu,\epsilon}(t)\|^{2}}{2}&+\nu\int_{0}^{t}\|\nabla\vne\|^{2}\,d\tau
      + \epsilon\int_{0}^{t}\|\nabla\rho^{\nu,\epsilon}\|^{2}\,d\tau
      \\
      &+\nu\int_{0}^{t}\int_{\Gamma}\vne\cdot(\nabla n)^T\cdot\vne\,dS
      d\tau\leq \frac{\|v_0^\nu\|^{2}+\|\rho_0^{\nu}\|^{2}}{2}.
    \end{aligned}
  \end{equation*}
  As $\epsilon$ goes to zero, (and by standard results of compactness)
  one can find sub-sequences such that
  \begin{equation*}
    \begin{aligned}
      &\vne\overset{*}{\rightharpoonup} \vn \qquad \text{and}\qquad
      \rho^{\nu,\epsilon}\overset{*}{\rightharpoonup} \rho^{\nu}\qquad
      &\text{in }L^\infty(0,T;L^2(\Omega)),
      \\
      &\vne\rightharpoonup \vn \qquad&\text{in } L^2(0,T;H^1(\Omega)),
      \\
      &\rho^{\nu,\epsilon}\rightharpoonup \rho^\epsilon
            \qquad \text{and}\qquad
%\qquad&\text{in } L^2(0,T;L^2)
 %     \\
       \vne\to \vn \qquad &\text{in } L^2(0,T;L^2(\Omega)),
    \end{aligned}
  \end{equation*}
  for some $\vn\in L^\infty(0,T;L^2(\Omega))\cap L^2(0,T;H^1(\Omega))$
  and $\rho^{\nu}\in L^\infty(0,T;L^2(\Omega))$. By standard arguments
  it turns out that $(\vn,\rho^\nu)$ is a weak solution
  to~\eqref{eq:NSB-Navier}. In particular, being $\vne$ strongly
  convergent and  $\rho^{\nu,\epsilon}$ weakly convergent in
  $L^2(0,T;L^2(\Omega))$, one has
  \begin{equation*}
    \int_0^T\int_\Omega(\vne\cdot\nabla \psi)\,\rho^{\nu,\epsilon}\,dx d\tau\to
    \int_0^T\int_\Omega(\vn\cdot\nabla \psi)\,\rho^{\nu}\,dx d\tau, 
  \end{equation*}
  and also
  \begin{equation*}
    \epsilon\int_0^T\int_\Omega\nabla\rho^{\nu,\epsilon}\nabla\psi\,dx d\tau\to 0,
  \end{equation*}
  due to the uniform bound for
  $\epsilon\int_0^T\int_\Omega|\nabla\rho^{\nu,\epsilon}|^2dx d\tau$.
  This implies that the $(\vn,\rho^\nu)$ is a weak solution in the
  sense of Definition~\ref{def:weak-solutionB} and in particular that
  energy inequality~\eqref{eq:Energy-NSB-NSB} is satisfied.
\end{proof}

From the existence result we prove now the last result of this paper.
\begin{theorem}
  \label{thm:3}
  Let $\Omega$ be a bounded smooth open set in $\R^3$, let
  $v_0^E\in H^3(\Omega)$, be a divergence-free vector field tangential to
  the boundary, let also $\rho_0^E\in H^3(\Omega)$. Let us suppose
  that the initial data satisfy the following conditions at the boundary:
  \begin{equation}
    \label{eq:zeroidB}
    \begin{aligned}
      \omega_{0}^E(x)=0\qquad\qquad \forall\,x\in\Gamma,
      \\
      \nabla\rho_{0}^E(x)=0\qquad\qquad \forall\,x\in\Gamma.
    \end{aligned}
 \end{equation}
 Let $\ue,\rho^E\in C([0,T];H^3(\Omega))$ be the unique solution of
 the Euler Boussinesq equations~\eqref{eq:EulerB}, with initial datum
 $(v_0^E,\rho^E_0)$ and defined in some interval $[0,T]$.

  Let $(\vn,\rho^\nu)$ be a weak solution of the Boussinesq
  equations~\eqref{eq:NSB-Navier} with a divergence free and
  tangential to the boundary initial datum $u_0^n\in L^2(\Omega)$ and
  with $\rho_0^\nu\in L^2(\Omega)$ such that 
  \begin{equation*}
    \|u^\nu_0-u_0^E\|=\mathcal{O}(\nu)\qquad\text{and}\qquad
    \|\rho^\nu_0-\rho_0^E\|=\mathcal{O}(\nu).
  \end{equation*} 
Then,
  \begin{equation*}
    \begin{aligned}
      &
      \sup_{t\in[0,T]}\|\un-\ue\|^2_{L^{\infty}(0,T;L^{2}(\Omega))}=\mathcal{O}(\nu^2),
      \qquad
      \sup_{t\in[0,T]}\|\rho^\nu-\rho^E\|^2_{L^{\infty}(0,T;L^{2}(\Omega))}=\mathcal{O}(\nu^2),
      \\
      &\qquad\text{and}\qquad
      \\
      &
      \|\nabla\un-\nabla\ue\|^2_{L^{2}(0,T;L^{2}(\Omega))}=\mathcal{O}(\nu).
   \end{aligned}
 \end{equation*}
\end{theorem}
\begin{proof}
  The proof is based on the same techniques employed before. We
  multiply \eqref{eq:NSB-Navier} by $(v^E,\rho^E)$ and with suitable
  integration by parts we get
  \begin{equation}
    \label{eq:energy-NSB-EB}
    \begin{aligned}
      &\int_{\Omega}\vn(t)\ue(t)+\rho^\nu(t)\rho^E(t)\,dx+
      \int_{0}^{t}\int_{\Omega}\Big(\nu\nabla \vn\nabla \ve
      +(\vn\cdot\nabla)\, \vn \ve
      \\
      &+(\vn\cdot\nabla)\, \rho^{\nu} \rho^E\Big)\,dx
      d\tau-\int_{0}^{t}\int_{\Omega} \rho^\nu\rho_t^E+\vn\ve_{t}\,dx
      d\tau +\nu\int_{0}^{t}\int_{\Gamma}\vn(\nabla n)^T \ve\,dSd\tau
      \\
      &=\int_{\Omega}v_{0}^E v_0^\nu+\rho_{0}^E \rho_0^\nu\,dx.
    \end{aligned}
  \end{equation}
We then obtain from the equation for $(\ve,\rho^E)$
  \begin{equation}
    \label{eq:energy-EB-NSB}
\int_{0}^{t}\int_{\Omega}\big(\ve_t
    \vn+ \rho^E_t
    \rho^{\nu}+(\ve\cdot\nabla)\, \ve \vn+(\ve\cdot\nabla)\, \rho^E \vn\big)\,dx d\tau
    =\int_{\Omega}v_{0}^E v_0^\nu+\rho_{0}^E \rho_0^\nu\,dx.
  \end{equation}
 Next, by multiplying the Euler Boussinesq equations by $\ue$ and by the usual integrations
  by parts we get the energy conservation
  \begin{equation}
    \label{eq:energy-EB-EB}
    \frac{\|\ve(t)\|^2+\|\rho^E(t)\|^2}{2}=\frac{\|v_0^E\|^2+\|\rho_0^E\|^2}{2}.
  \end{equation}
Then, by adding together~\eqref{eq:Energy-NSB-NSB}-\eqref{eq:energy-EB-EB}
and subtracting~\eqref{eq:energy-NSB-EB}-\eqref{eq:energy-EB-NSB}, we
get
  \begin{equation*} 
%    \label{p5B} 
    \begin{aligned}
 &     \frac{\|v(t)\|^{2}+\|\rho(t)\|^{2}}{2}+
      \nu\int_{0}^{t}\int_{\Omega}\nabla\vn\nabla v\,dx d\tau
      +\int_{0}^{t}\int_{\Omega}(v\cdot\nabla)\,\ve v\,dx d\tau 
      \\
      &
      \qquad   -\int_{0}^{t}\int_{\Omega}(v\cdot\nabla)\,\rho^E \rho\,dx d\tau +\nu
      \int_{0}^{t}\int_{\Gamma}\vn(\nabla n)^T v\,dS d\tau
            \leq      \frac{\|v(0)\|+\|\rho(0)\|^{2}}{2},
    \end{aligned}
    \end{equation*}
where
\begin{equation*}
  v:=\vn-\ve\qquad\text{and}\qquad \rho:=\rho^\nu-\rho^E.
\end{equation*}
With the same manipulations employed in the previous section we get
\begin{equation*}
  \begin{aligned}
    &\frac{\|v(t)\|^{2}+\|\rho(t)\|^2}{2}+\nu\int_{0}^{t}\|\nabla
    v\|^{2} d\tau\leq
    \frac{\|v(0)\|^{2}+\|\rho(0)\|^{2}}{2}-\nu\int_{0}^{t}\int_{\Gamma}v\,
    (\nabla n)^T v\,dS d\tau
  % -\nu\int_{0}^{t}\int_{\Gamma}\ue\cdot(\nabla n)^T\cdot u\,d Sd\tau
  \\
  &\qquad-\int_{0}^{t}\int_{\Omega}(v\cdot\nabla)\,\ve v+(v\cdot\nabla)\,\rho^E \rho\,dx
  d\tau-\nu\int_{0}^{t}\int_{\Omega}\Delta\ve v\,dx d\tau
\\
&\qquad+\nu\int_{0}^{t}\int_{\Gamma}(\curl\ve\times n)\, v\,dS d\tau.
\end{aligned}
\end{equation*}
We then estimate most of the terms as before with
  \begin{align*}
%    &   \nu\left|\int_{\Gamma}(\curl\ve\times n) v\,dS\right|\leq
 %   C\nu\|v\|^{\frac{1}{2}}\|\nabla   v\|^{\frac{1}{2}}
  % \leq C\,\nu^{\frac{3}{2}}+C\,\|v\|^{2}+\frac{\nu}{2}\|\nabla v\|^{2},
    & \left|\int_{\Omega}(v\cdot\nabla)\,\ve\,
      v+(v\cdot\nabla)\,\rho^E \rho\,dx\right|\leq
    C\big(\|u(t)\|^2+\|\rho(t)\|^2\big),
    \\
    \\
    & \nu\left|\int_{\Gamma} v\cdot( \nabla n)^T\cdot v\,dS\right|
%C \,\nu\|v\|_{\Gamma}^{2}\leq C\,\nu    \|v\|\|\nabla u\| 
\leq
    C\,    \nu\|v\|^{2}+\frac{\nu}{2}\|\nabla v\|^{2},
    \\
    \\
    &\nu\left|\int_{\Omega}\Delta \ve v\,dx\right|\leq
    C\big(\|v\|^{2}+\nu^2\big),
\end{align*}  
where we used that $\ve\in C([0,T];H^3(\Omega))$. 

To handle the last integral involving $\curl\ve$ on $\Gamma$, we
observe that the equation for the vorticity for the Euler-Boussinesq
system implies, along path-lines,
\begin{equation*}
    \curl \ve(X(\alpha,t),t)=\nabla_\alpha X(\alpha,t)\curl \ve(\alpha,0)-\int_0^t\curl(
    \rho^E e_3)(X(\alpha,\sigma),\sigma)\,d\sigma,
  \end{equation*}
  hence now $\curl v^E_0(x)=0$ on $\Gamma$ is not enough to have
  $\curl v^E=0$ on the boundary for all positive times and a control
  also on $\curl(\rho^E e_3)$ at the boundary is needed. We then
  observe that $\rho$ is transported by the velocity $\ve$ (it solves the
  equation of continuity), hence
\begin{equation*}
  \rho^E(X(\alpha,t),t)=\rho^E_0(\alpha).
\end{equation*}
Consequently, we get the following evolution equation for the
gradient of $\rho$:
\begin{equation*}
\nabla\rho^E X(\alpha,t),t)\,\nabla_\alpha X(\alpha,t)=\nabla \rho_0^E(\alpha),
%  \sum_{j=1}^3\frac{\rho^E(X(\alpha,t),t)}{\partial
%    x_j}\frac{\partial X_j(\alpha,t)}{\partial
%    \alpha_i}=\frac{\partial\rho^E_0(\alpha)}{\partial \alpha_i},
\end{equation*}
hence 
\begin{equation*}
 \nabla \rho^E(X(\alpha,t),t)=\nabla \rho_0^E(\alpha)\big[\nabla_\alpha X(\alpha,t)\big]^{-1}.
\end{equation*}
%and consequently, 
Since the matrix $\nabla_\alpha X$ is non-singular and since
path-lines starting at the boundary remain at the boundary, it follows
that a sufficient condition to have $\curl(\rho^E e_3)(x,t)=0$ for all
$x\in \Gamma$ and for all $t\in[0,T]$ is that of
asking~\eqref{eq:zeroidB}. Under the above assumptions the term
$\int_{\Gamma}(\curl\ve\times n)\, v\,dS$ vanishes identically, hence
we arrive at the inequality
\begin{equation*}
  \|u(t)\|^{2}+\|\rho(t)\|^{2}+\nu\int_{0}^{t}\|\nabla u\|^{2}\,d\tau\leq
  \|u(0)\|^{2}+\|\rho(0)\|^{2}+
  C\Big[\int_{0}^{t}\|u(\tau)\|^{2}\,d\tau+\nu^{2}\Big],
  \end{equation*}
from which we have the thesis by applying the Gronwall lemma.
\end{proof}
\def\ocirc#1{\ifmmode\setbox0=\hbox{$#1$}\dimen0=\ht0 \advance\dimen0
  by1pt\rlap{\hbox to\wd0{\hss\raise\dimen0
  \hbox{\hskip.2em$\scriptscriptstyle\circ$}\hss}}#1\else {\accent"17 #1}\fi}
  \def\polhk#1{\setbox0=\hbox{#1}{\ooalign{\hidewidth
  \lower1.5ex\hbox{`}\hidewidth\crcr\unhbox0}}} \def\cprime{$'$}

%\bibliographystyle{plain}
%\bibliography{../../../Dropbox/buro/raccoglitore/lista,%
%../../../Dropbox/buro/raccoglitore/altri,../../../Dropbox/buro/raccoglitore/libri}
\end{document}